\documentclass[11pt, reqno, a4]{article}
\usepackage{amsmath,amssymb,amsfonts,amsthm}
\usepackage{color}

\usepackage[colorlinks=true,linkcolor=blue,citecolor=blue,urlcolor=blue,pdfborder={0 0 0}]{hyperref}
\usepackage{cleveref}

\usepackage{caption}
\usepackage{thmtools}

\usepackage{mathrsfs}

\crefname{theorem}{Theorem}{Theorems}
\crefname{thm}{Theorem}{Theorems}
\crefname{lemma}{Lemma}{Lemmas}
\crefname{lem}{Lemma}{Lemmas}
\crefname{remark}{Remark}{Remarks}
\crefname{prop}{Proposition}{Propositions}
\crefname{defn}{Definition}{Definitions}
\crefname{corollary}{Corollary}{Corollaries}
\crefname{conjecture}{Conjecture}{Conjectures}
\crefname{question}{Question}{Questions}
\crefname{chapter}{Chapter}{Chapters}
\crefname{section}{Section}{Sections}
\crefname{figure}{Figure}{Figures}

\theoremstyle{plain}
\newtheorem{thm}{Theorem}[section]

\newtheorem{lemma}[thm]{Lemma}

\newtheorem{corollary}[thm]{Corollary}

\newtheorem{question}[thm]{Question}
\theoremstyle{definition}

\theoremstyle{remark}
\newtheorem{remark}[thm]{Remark}

\numberwithin{equation}{section}

\newcommand{\E}{\mathbb E}

\newcommand{\R}{\mathbb R}
\newcommand{\Z}{\mathbb Z}
\newcommand{\N}{\mathbb N}

\usepackage[margin=1in]{geometry}
\usepackage{extarrows}
\usepackage{commath}
\usepackage{mathtools}
\newcommand{\eps}{\varepsilon}
\usepackage{bbm}
\usepackage{setspace}
\usepackage{enumitem}
\usepackage{tikz}
\usepackage{bm}

\newcommand{\Aut}{\operatorname{Aut}}

\title{Statistical physics on a product of trees}
\author{Tom Hutchcroft\footnote{Statslab, DPMMS, University of Cambridge, and Trinity College, Cambridge}}

\begin{document}
\maketitle

\begin{abstract}
Let $G$ be the product of finitely many trees $T_1\times T_2 \times \cdots \times T_N$, each of which is regular with degree at least three. We consider Bernoulli bond percolation and the Ising model on this graph, giving a short proof that the model undergoes a second order phase transition with mean-field critical exponents in each case. The result concerning percolation recovers a result of Kozma (2013), while the result concerning the Ising model is new. 

We also present a new proof, using similar techniques, of a lemma of Schramm concerning the decay of the critical two-point function along a random walk, as well as some generalizations of this lemma.
\end{abstract}

\section{Introduction}

In \cite{kozma2011percolation}, Gady Kozma proved that the \emph{triangle condition} holds for critical Bernoulli bond percolation on the Cartesian product of two $3$-regular trees. (His result also follows from our recent work \cite{Hutchcroftnonunimodularperc}.) 
The triangle condition is a well-known signifier of mean-field critical behaviour, originally introduced by Aizenman and Newman \cite{MR762034}, and can be used to deduce that various critical exponents take their mean-field values. For example, it implies that the expected cluster volume at $p=p_c-\eps$ scales like $\eps^{-1}$ \cite{MR762034}, that the density of the infinite cluster at $p=p_c+\eps$ scales like $\eps$ \cite{MR1127713}, and that the probability at criticality that the cluster of the origin has volume at least $n$ scales like $n^{-1/2}$ \cite{MR1127713}. See e.g. \cite{heydenreich2015progress} and \cite[Sections 1.3 and 7]{Hutchcroftnonunimodularperc} for an overview. For background on percolation see e.g.\ \cite{grimmett2010percolation,LP:book,1707.00520}. 

 In this note, we give a short and elementary proof of (a slight generalisation of) his result, together with the existence of a non-uniqueness phase for percolation on the same graph,   using ideas similar to those used in \cite{Hutchcroft2016944,Hutchcroftnonunimodularperc,1709.10515}. 
Our proof uses only a few basic properties of percolation and is also applicable to e.g.\ the Ising model, see \cref{sec:Ising}. We denote by $\tau_p(x,y)$ the probability that $x$ is connected to $y$ in Bernoulli bond percolation with retention probability $p$. As usual, $p_c$ and $p_u$ denote the thresholds for the appearance of an infinite cluster and a unique infinite cluster respectively. 

\begin{thm}[Non-uniqueness]
\label{thm:pcpu}
Let $G = T_1 \times T_2 \times \cdots \times T_N$ be the Cartesian product of finitely many trees $T_i$, each of which is regular with some degree $k_i \geq 3$. 
Then $p_c(G)<p_u(G)$. 
\end{thm}

We use $0$ to denote an arbitrarily chosen root vertex of $G$. We have not optimized the upper bound on $\nabla_{p_c}$ appearing below. 

\begin{thm}[Triangle condition]\label{thm:main}
Let $G = T_1 \times T_2 \times \cdots \times T_N$ be the Cartesian product of finitely many trees $T_i$, each of which is regular with some degree $k_i \geq 3$. 
Then 
\[\nabla_{p_c} := \sum_{x,y} \tau_{p_c}(0,x)\tau_{p_c}(x,y)\tau_{p_c}(y,0) \leq \prod_{i=1}^N \frac{(k_i-1)^3}{(\sqrt{k_i-1}-1)^6}
< \infty.\]
\end{thm}

\cref{thm:pcpu,thm:main} both follow from the following estimate on critical connectivity probabilities.
 Given $x,y\in V$, we write $\mathbf{d}(x,y)=(d_i(x_i,y_i))_{i=1}^N$ for the vector of distances between the coordinates of $x$ and $y$. We also write $\bm{\delta}$ for the vector $\bm{\delta} = (\log(k_i-1))_{i=1}^N$.

\begin{thm}\label{thm:pcestimate} Let $G = T_1 \times T_2 \times \cdots \times T_N$ be the Cartesian product of finitely many trees $T_i$, each of which is regular with some degree $k_i \geq 3$. Then
\begin{equation}
\label{eq:pcestimate}
\tau_{p_c}(x,y) \leq \prod_{i=1}^N(k_i-1)^{-d_i(x_i,y_i)} = \exp\left[ - \bm{\delta} \cdot \mathbf{d}(x,y) \right].
\end{equation}
\end{thm}

\begin{remark}
The inequality \eqref{eq:pcestimate} is an equality in the case $N=1$. 
In \cite{kozma2011percolation}, Kozma proved the slightly weaker inequality
\vspace{-0.25cm}
\[
\tau_{p_c}(x,y) \preceq  \|\mathbf{d}(x,y)\|^3 \exp\left[ - \bm{\delta} \cdot \mathbf{d}(x,y) \right].
\]
in the case $N=2$, $d_1=d_2=3$.
\end{remark}

\begin{remark}
Kozma's proof relied upon an unpublished lemma of Schramm giving an upper bound on the probability that the two endpoints of a random walk are connected in critical percolation. We give a new proof of this estimate in \cref{sec:Schramm} using techniques similar to those used to prove \cref{thm:pcestimate}.
\end{remark}

\begin{remark}
The proof of \cref{thm:pcpu} yields explicit lower bounds on $p_u-p_c$. In particular, it shows that
\[
p_u-p_c \geq \frac{1-p_c}{2 \sum_{i=1}^N \sqrt{k_i-1}} \prod_{i=1}^N \frac{(\sqrt{k_i-1}-1)^2}{k_i-1}.
\]
Again, this constant has not been optimized.
\end{remark}

\begin{remark}
If one instead considers \emph{anisotropic} percolation, in which there is a different retention probability associated to each tree in the product $G=T_1\times T_2 \times \cdots \times T_N$, our proof shows that the estimate \eqref{eq:pcestimate} holds uniformly along the entire critical surface for the existence of an infinite cluster. It follows that the triangle sum is uniformly bounded on the existence critical surface and that the existence critical surface and the uniqueness critical surface are bounded away from each other.
\end{remark}

\begin{remark}
If $p_{c,\lambda}$ is defined as in \cite{Hutchcroftnonunimodularperc}, then the proof of \cref{thm:pcestimate} shows more generally that if $G$ is a connected, locally finite graph, and $\Gamma$ is a transitive nonunimodular subgroup of $\Aut(G)$ with modular function $\Delta$, then we have the bound
\[\tau_{p_{c,\lambda}}(x,y)=\tau_{p_{c,1-\lambda}}(x,y) \leq \Delta^{\lambda}(x,y)\]
for every $\lambda \in \R$ and $x,y \in V$. In particular, if  $G = T_1 \times T_2 \times \cdots \times T_N$ is a product of finitely many trees $T_i$, each of which is regular with some degree $k_i \geq 3$, and $\Gamma$ is the group of automorphisms of $G$ fixing some specified end of each tree $T_1,\ldots,T_N$, then it follows that 
\[\tau_{p_{c,\lambda}}(x,y) \leq \prod_{i=1}^N(k_i-1)^{-\max\{\lambda,1-\lambda\} d_i(x_i,y_i)} = \exp\left[ - \max\{\lambda,1-\lambda\}\; \bm{\delta} \cdot \mathbf{d}(x,y) \right] \qquad \]
for every $\lambda \in \R$ and $x,y \in V$, generalizing \cref{thm:pcestimate}. Using this estimate one can easily prove that $p_{c,\lambda}$ is a strictly increasing function of $\lambda$ on $(-\infty,1/2]$, verifying \cite[Conjecture 8.4]{Hutchcroftnonunimodularperc} for this example.
\end{remark}

The large amount of symmetry enjoyed by a product of trees would seem to make it an excellent example with which to develop a deeper understanding of percolation at the non-uniqueness threshold. 

\begin{question}
It follows from the work of Peres \cite{MR1770624} that there is not a unique infinite cluster at $p_u$ on $G=T^N$, where $T$ is a $k$-regular tree for some $k\geq 3$ and $N\geq 2$. Does the triangle condition hold at $p_u$ in this example? Do we have that
\[
\tau_{p_u}(x,y) \leq C \exp\left[ - \frac{1}{2}\bm{\delta} \cdot \mathbf{d}(x,y) \right]
\]
for some constant $C$? What if $N$ is large?
Might we even have that
\[
\tau_{p_u}(x,y) \leq C \|\mathbf{d}(x,y)\|^{-\gamma_N} \exp\left[ - \frac{1}{2}\bm{\delta} \cdot \mathbf{d}(x,y) \right]
\]
for some constant $C$ and some $\gamma_N >0$? 
\end{question}

\subsection{The Ising model}
\label{sec:Ising}

The proof of \cref{thm:main,thm:pcestimate} is not particularly specific to percolation: It relies only on the positive associativity  property (i.e.\ the Harris-FKG inequality), the monotonicity and left-continuity in $p$  of the random subgraph measures under consideration,  and on the sharpness of the phase transition (in the sense that the susceptibility is finite below $p_c$). As a result, it can also be applied immediately to the Ising model in the same setting (equivalently, the random cluster model with $q=2$), for which sharpness was established by Aizenman, Barsky, and Fernandez \cite{MR894398} (see also \cite{duminil2015new}). Here, the relevant signifier of mean-field behaviour at criticality is the convergence of the \emph{bubble diagram} rather than the triangle diagram.

We use $\langle \cdot \rangle_\beta$ to denote expectations with respect to the free-boundary-condition Ising model with inverse temperature $\beta$ (with no external field) and use $\langle \cdot \rangle_{\beta,h}$ to denote expectations with respect to the free-boundary-condition Ising model with inverse temperature $\beta$ and external field $h$. For background on the Ising model see e.g.\ \cite{1707.00520}.

\begin{thm}\label{thm:Isingmain}
Let $G = T_1 \times T_2 \times \cdots \times T_N$ be the Cartesian product of finitely many trees $T_i$, each of which is regular with some degree $k_i \geq 3$, and consider the ferromagnetic Ising model on $G$. 
Then
\[\langle \sigma_x \sigma_y \rangle_{\beta_c} \leq \prod_{i=1}^N(k_i-1)^{-d_i(x_i,y_i)} = \exp\left[ - \bm{\delta} \cdot \mathbf{d}(x,y) \right] \]
for every $x,y \in V$, and hence
\[B_{\beta_c} = \sum_{x} \langle \sigma_0 \sigma_x \rangle_{\beta_c}^2 \leq \prod_{i=1}^N \frac{(k_i-1)^2}{(\sqrt{k_i-1}-1)^4} < \infty.\]
\end{thm}

The fact that the following corollary can be deduced from \cref{thm:Isingmain} is essentially contained in the papers \cite{MR588470,MR678000,MR894398,MR857063}; see also \cite[Section 4.2]{MR1833805}.  
We write `$f(x) \asymp g(x)$ as $x\nearrow x_0$' to mean that $\limsup_{x\uparrow x_0} f(x)/g(x)<\infty$ and $\liminf_{x\uparrow x_0} f(x)/g(x) >0$. The meaning of `$f(x)\asymp g(x)$ as $x \searrow x_0$' is similar. 
 % use $\asymp$ for an equality that holds up to multiplication by a postitive function that is bounded away from $0$ and $\infty$ in the vicinity of the relevant limit point.

\begin{corollary}[Mean-field critical exponents]
\label{corollary:Isingcriticalexponents}
Let $G = T_1 \times T_2 \times \cdots \times T_N$ be the Cartesian product of finitely many trees $T_i$, each of which is regular with some degree $k_i \geq 3$, and consider the ferromagnetic Ising model on $G$. Then we have that
\vspace{0em}
\begingroup
\addtolength{\jot}{0.5em}
\begin{align}
\chi_\beta &:= \langle \sigma_0 \sigma_x \rangle_{\beta} \asymp (\beta_c-\beta)^{-1}  &\beta &\nearrow \beta_c
\label{exponent:Isingsusceptibility}
\\
M_{\beta_c,h}&:= \langle \sigma_0 \rangle_{\beta_c,h}  \hspace{0.045cm}\asymp h^{1/3} & h&\searrow 0
\label{exponent:Isingcritmag}
\\
\lim_{h\downarrow 0} M_{\beta,h} &:= \lim_{h\downarrow0}\langle\sigma_0\rangle_{\beta,h} \hspace{0.085cm}  \asymp (\beta-\beta_c)^{1/2} & \beta&\searrow \beta_c.
\label{exponent:Isingsupermag}
% \\
% M_{p_c,h}(v) &\asymp h^{1/2} &h&\searrow 0\\
\end{align}
\endgroup
In particular, the spontaneous magnetization is continuous at $\beta_c$.
\end{corollary}

\begin{remark}
Unlike for percolation, it is not yet known that the spontaneous magnetization is continuous at $\beta_c$ for the Ising model on every transitive nonamenable graph, even in the unimodular case. (The method of Benjamini, Lyons, Peres, and Schramm \cite{BLPS99b} only implies that the free FK-Ising model does not have any infinite clusters at criticality.)  Previously, continuity of the spontaneous magnetization had been established for $\Z^d$ with $d\geq 2$ along with some other Euclidean lattices \cite{MR0010315,MR0051740,MR857063,MR3306602}, and for amenable quasi-transitive graphs of exponential growth \cite{1606.03763}.
\end{remark}
\begin{remark}
In recent work by Duminil-Copin, Tassion, and Raoufi \cite{1705.03104,1705.07978}, a general methodology has been established to prove exponential decay of connectivity probabilities for many subcritical models with the positive associativity property (e.g.\ the FK-random cluster model for $q\geq 1$). Michael Aizenman has recently announced a proof, using related methods, that these models also have \emph{finite susceptibility} in their subcritical phases. This will allow one to deduce analogues of \cref{thm:pcestimate} and \cref{thm:Isingmain} for these models on a product of trees via our methods. We believe that similar methods should also enable one to analyze e.g.\ Voronoi percolation in hyperbolic spaces (which, like trees, are distance-transitive), and plan to address this in future work.
\end{remark}

\section{Proof}

The most important input to the proof of \cref{thm:pcestimate} is that the phase transition is \emph{sharp}, i.e., that
\[
\chi_p := \sum_{x\in V} \tau_p(0,x) <\infty \qquad \text{ for every $p<p_c$.}
\]
This was originally proven for all transitive graphs  by Aizenman and Barsky~\cite{aizenman1987sharpness}. A beautiful new proof was recently obtained by Duminil-Copin and Tassion \cite{duminil2015new}.

We will also make crucial use of Fekete's Lemma \cite{MR1544613} in the following form: If $(a_n)_{n\geq0}$ is a sequence of positive real numbers satisfying the supermultiplicative estimate
$a_{n+m} \geq a_n a_m$, then 
\[ \lim_{n\to\infty} a_n^{1/n} = \sup_{n\geq1} a_n^{1/n} \in (0,\infty].\]
In particular, the limit on the left exists, and $a_n \leq (\lim_{m\to\infty} a_m^{1/m})^n$ for every $n\geq0$.

\begin{proof}[Proof of \cref{thm:pcestimate}]
Recall that for any two vertices $x, y\in V$, the connection probability $\tau_p(x,y)$ can be written as the supremum of the continuous functions $\tau_p^r(x,y)$ giving the probability that $x$ is connected to $y$ by a path of length at most $r$, so that $\tau_p(x,y)$ is lower-semicontinuous in $p$. Since $\tau_p(x,y)$ is an increasing function of $p$, it follows that it is left-continuous in $p$. Thus, it suffices to prove the claim for all $p<p_c$. 

Observe that for any two vertices $x$ and $y$ of $G$ and $p\in [0,1]$, the connection probability $\tau_p(x,y)$ depends only on $p$ and on the vector of distances $\mathbf{d}(x,y):=(d_i(x_i,y_i))_{i=1}^N$. (Indeed, the isomorphism class of the doubly-rooted graph $(G,x,y)$ depends only on this vector of distances.) For each vector of non-negative integers $\mathbf{n}=(n_i)_{i=1}^N$, we define 
\[
V_\mathbf{n}(x) = \left\{y\in V : \mathbf{d}(x,y)=\mathbf{n}\right\}
\]
and
\[
\nu_p(\mathbf{n}) = \tau_p(x,y) \qquad y\in V_\mathbf{n}(x).
\]
Given $m\geq 1$, we write $m \mathbf{n}= (mn_i)_{i=1}^N$. Observe that if $r,\ell\geq0$ then there exists $y,z\in V$ such that $\mathbf{d}(x,y)=r \mathbf{n}$,  $\mathbf{d}(y,z)=\ell \mathbf{n}$ and $\mathbf{d}(x,z)=(r+\ell)\mathbf{n}$. Indeed, simply choose $z \in V_{(r+\ell)\mathbf{n}}(x)$ arbitrarily and take $y_i$ to be the $r$th vertex on the geodesic in $T_i$ from $x_i$ to $z_i$ for each $1\leq i \leq N$. Thus, it follows from the Harris-FKG inequality that the submultiplicative estimate
 % We claim that the inequality
\[\nu_p((r+\ell)\mathbf{n}) \geq \nu_p(r \mathbf{n})\nu_p(\ell \mathbf{n}) \]
holds for every $k,\ell \geq 0$. 
% Suppose that $$
If $p<p_c$ then it follows by Fekete's Lemma that
\[
\nu_p(\mathbf{n})\leq \lim_{r\to\infty}\nu_p(r\mathbf{n})^{1/r} \leq \liminf_{r\to\infty} \left[\frac{\chi_p}{|V_{r\mathbf{n}}(x)|}\right]^{1/r} = \liminf_{k\to\infty} |V_{r\mathbf{n}}(x)|^{-1/r}.\]
% and hence that
% \[
% \lim_{k\to\infty} \nu_p(k\mathbf{n}) \leq \prod_{i=1}^N(d_i-1)^{-kn_i}
% \]
The result now follows by observing that
\[
 |V_{r\mathbf{n}}(x)| \geq \prod_{i=1}^N(k_i-1)^{rn_i}
\]
and hence that
\[
\nu_p(\mathbf{n})\leq \lim_{r\to\infty} |V_{r\mathbf{n}}(x)|^{-1/r} = \prod_{i=1}^N(k_i-1)^{-n_i}. \qedhere
\]
\end{proof}

For each $1\leq i \leq N$, let $\xi_i$ be a fixed end of $T_i$. The \textbf{parent} of a vertex $x_i\in T_i$ is the unique neighbour of $x_i$ that is closer to $\xi_i$ than $x_i$ is. We call the other vertices of $x_i$ its \textbf{children}.
Given this information, we can partition each of the trees $T_i$ into \textbf{levels} $(L_{i,n})_{n\in \Z}$ such that for every $n\in \Z$, every vertex in $L_n$ has its parent in $L_{n+1}$ and its children in $L_{n-1}$. This partition is unique up to re-indexing. 
% $for every 
Let $h_i(x_i,y_i)$ denote the \textbf{height difference} between two vertices $x_i,y_i\in T_i$, so that $h_i(x_i,y_i)= k$ if and only if there exists $n\in \Z$ such that $x_i \in L_{i,n}$ and $y_i\in L_{i,n+k}$. Note that this definition does not depend on the choice of index used when defining the levels $L_{i,n}$. 
Define\footnote{$\Delta(x,y)$ is equal to the modular function of $G$ with respect to the group of automorphisms $\Gamma_1 \times \cdots \Gamma_N \subseteq \operatorname{Aut}(T_1\times\cdots\times T_n)$, where $\Gamma_i$ is the group of automorphisms of $T_i$ fixing the end $\xi_i$.}  $\Delta:V^2\to(0,\infty)$ by
 \[\Delta(x,y) = \prod_{i=1}^N(k_i-1)^{h_i(x_i,y_i)} = \exp\left[ \delta \cdot \mathbf{h}(x,y) \right].\]
 Note that $\Delta(x,z)=\Delta(x,y)\Delta(y,z)$ and that $\Delta(x,y)=\Delta(y,x)^{-1}$ for every three vertices $x,y,z$.
We define the \textbf{critically tilted susceptibility} to be
\[\chi_{p,1/2} = \sum_x \tau_p(0,x) \Delta(0,x)^{1/2}.\]

\cref{lem:tilttotri,lem:opencondition}, below, are special cases of \cite[Lemma 7.1 and Proposition 1.9]{Hutchcroftnonunimodularperc} respectively.
\begin{lemma}
\label{lem:tilttotri}
$\nabla_p \leq \bigl(\chi_{p,1/2}\bigr)^3$.
In particular, if $\chi_{p,1/2}$ is finite then so is $\nabla_p$.
\end{lemma}

\begin{proof}
We have that
\begin{align*}
\nabla_p = \sum_{x,y} \tau_p(0,x) \tau_p(x,y) \tau_p(y,0) \leq \sum_{x,y,z} \tau_p(0,x) \tau_p(x,y) \tau_p(y,z) \Delta(0,z)^{1/2}\end{align*}
and using the identity $\Delta(0,z)^{1/2}=\Delta(0,x)^{1/2}\Delta(x,y)^{1/2}\Delta(y,z)^{1/2}$ yields that
\[\nabla_p \leq \sum_x \tau_p(0,x)\Delta(0,x)^{1/2} \sum_y \tau_p(x,y)\Delta(x,y)^{1/2} \sum_z \tau_p(y,z) \Delta(y,z)^{1/2} = \left(\chi_{p,1/2}\right)^3.  \qedhere
\]
\end{proof}

\begin{lemma}
\label{lem:opencondition}
The set $\{p \in [0,1] : \chi_{p,1/2} <\infty \}$ is open in $[0,1]$.
\end{lemma}
% In fact the proof can be modified to yield the stronger statement that $\beta_p$ is right-continuous on the set $\{ p : \beta_p > 1/2\}$.
% Let 
% \[\chi_{p_c}(t) \]
% \[ \diamondsuit_p = \sum_x \tau_p(0,x) \Delta(0,x)^{1/2}. \]
% Then 
% \begin{enumerate}
% \item The set $\{p \in [0,1] : \diamondsuit_p < \infty\}$ is open,
% \item $\diamondsuit_{p_c}<\infty$, and

% \item 
% $\diamondsuit_p < \infty \Rightarrow  \nabla_p < \infty$.
% \end{enumerate}
% \end{thm}

% Notice that parts 1 and 3 of the proof work for any graph with a nonunimodular transitive group of automorphisms. 

\begin{proof}
% Let 
% \[\diamondsuit_p = \sum_x \tau_p(0,x) \Delta(0,x)^{1/2}. \]
% % so that $\diamondsuit_p< \infty$ if and only if $\beta_p>1/2$.
% The value $1/2$ is special because we have the identity
% \[\sum_x \tau_p(0,x)\Delta(0,x)^b \mathbbm{1}(\Delta(0,x) \geq 1) = \sum_x \tau_p(0,x) \Delta(0,x)^{1-b} \mathbbm{1}(\Delta(0,x) \leq 1),\]
% so that for $b=1/2$
% \[\sum_x \tau_p(0,x)\Delta(0,x)^{1/2} \mathbbm{1}(\Delta(0,x) \geq 1) = \sum_x \tau_p(0,x) \Delta(0,x)^{1/2} \mathbbm{1}(\Delta(0,x) \leq 1).\]
% It follows that
% \[\frac{1}{2} \int_0^\infty \chi_p(t) e^{t/2} dt \leq \diamondsuit_p \leq \int_0^\infty \chi_p(t) e^{t/2} dt,\]
% and in particular that $\beta_p>1/2$ if and only if $\diamondsuit_p<\infty$. 
Let $p\in [0,1)$ be such that $\chi_{p,1/2}<\infty$, and let $0<\eps<1-p$.
Suppose each edge of $G$ is \textbf{open} with probability $p$ and, independently, \textbf{blue} with probability $\eps/(1-p)$. Note that the subgraph spanned by the open-or-blue edges is exactly $p+\eps$ percolation.
% \[\] 
 % Consider the standard coupling of $G[p]$ and $G[p+\eps]$. Say that an edge is open if its open in $G[p+\eps]$, and that its \textbf{blue} if its open in $G[p+\eps]$ but not $G[p]$.
  Let $\tilde \tau^i(x,y)$ be the probability that $x$ and $y$ are connected by an open-or-blue path that crosses each edge at most once and contains exactly $i$ blue edges,
  and let 
\[\tilde \chi^i = \sum_x \tilde\tau^i(0,x) \Delta(0,x)^{1/2}\]
so that 
\[\tau_{p+\eps}(x,y) \leq \sum_{i \geq 0} \tilde \tau^i(x,y) \quad \text{ and hence } \quad 
% and hence
% \[ 
\chi_{p+\eps,1/2} \leq \sum_{i \geq 0} \tilde \chi^i.\]
It follows from the BK inequality that
\[ \tilde \tau^{i+1}(0,x) \leq \frac{\eps}{1-p} \sum_y \tilde \tau^i(0,y) \sum_{z \sim y} \tau_p(z,x), \]
and hence that
\begin{align*}\tilde\chi^{i+1} &\leq \frac{\eps}{1-p} \sum_{y} \tilde \tau^i(0,y) \sum_{z\sim y} \sum_{x}\tau_p(z,x) \Delta(0,x)^{1/2}\\
&=  \frac{\eps}{1-p} \sum_{y} \tilde\tau^i(0,y) \Delta(0,y)^{1/2}  \sum_{z\sim y} \Delta(y,z)^{1/2} \sum_x \tau_p(z,x)  \Delta(z,x)^{1/2}  \\
& = \frac{\eps}{1-p}  \tilde \chi^i \chi_{p,1/2}\sum_{z\sim 0} \Delta(0,z)^{1/2}.
\end{align*}
% We claim that there exists a constant $C$ such that
Thus, it follows by induction that 
\[\tilde \chi^i \leq \chi_{p,1/2}\left( \frac{\eps}{1-p}  \chi_{p,1/2}\sum_{z\sim 0} \Delta(0,z)^{1/2} \right)^i,\]
and hence that
\begin{align*}\chi_{p+\eps,1/2} &\leq \frac{\chi_{p,1/2}}{1- \frac{\eps}{1-p} \chi_{p,1/2}\sum_{z\sim0} \Delta(0,z)^{1/2}} < \infty
\end{align*}
for every sufficiently small $\eps>0$, 
concluding the proof.
\end{proof}

\begin{lemma}
\label{lem:chipc}
$\chi_{p_c,1/2}<\infty$.
\end{lemma}

\begin{proof}
Observe that for every $1\leq i \leq N$, 
$m\geq 0$ and $n\geq 0$
 we have that
\[|\{ x_i \in V_i : h_i(0,x_i) = m-n, |x_i| = m+n\}| = \left.\begin{cases} 
(k_i-1)^n & m=0,\, n\geq 0\\
1 & m \geq 1,\, n=0\\
(k_i-1)^{n-1}(k_i-2) &  m\geq 1,\, n\geq 1. \end{cases}\right\} \leq (k_i-1)^n\]
% \[\ell=2m+|n|\]
% \[m=(\ell-|n|)/2\]
% \[
% m-n =\ell/2-|n|/2-n, n< 0 \rightarrow m = \ell/2-n/2
% \]
% \[V_{\mathbf{n},\mathbf{m}} = \prod_{i=1}^N (k_i-1)^{(n_i-m_i)/2}  \text{ if $\mathbf{n}-\mathbf{m}$ even} \]
Given $\mathbf{m}=(m_i)_{i=1}^N,\mathbf{n}=(n_i)_{i=1}^N, \in \N^N$, define 
\[
V_{\mathbf{m},\mathbf{n}}= \{ x \in V : h_i(0,x_i) = m_i-n_i, |x_i| = m_i+n_i\text{ for all $i=1,\ldots,N$}\},
\]
so that $|V_{\mathbf{m},\mathbf{n}}|\leq \prod_{i=1}^N(k_i-1)^{n_i}$.
%  It follows from this and \cref{thm:pcestimate} that for every $n_1,\ldots,n_N \in \Z$ we have
% \begin{multline*}\sum_x \tau_{p_c}(0,x) \mathbbm{1}(h_i(x_i) = n_i \; \forall 1\leq i \leq N)\\
%  = \sum_{m_1,\ldots,m_N \geq 0} \sum_x \tau_{p_c}(0,x) \mathbbm{1}(h_i(x_i) = n_i, |x_i|=|n_i|+2m_i \; \forall 1\leq i \leq N)\\
%  \leq \sum_{m_1,\ldots,m_N \geq 0} \sum_x  \prod_{i=1}^N (k_i-1)^{-(n_i \vee 0) -m_i}
% = C \prod_{i=1}^N (k_i-1)^{-(n_i \vee 0)}.  \end{multline*}
Thus, applying \cref{thm:pcestimate}, we can compute that
\begin{align}\chi_{p_c,1/2} &= \sum_{\mathbf{n}\in\N^N} \sum_{\mathbf{m}\in\N^N} \sum_{x\in V_{\mathbf{m},\mathbf{n}}}\tau_{p_c}(0,x) \prod_{i=1}^N(k_i-1)^{(m_i-n_i)/2} \leq \sum_{\mathbf{n}\in\N^N} \sum_{\mathbf{m}\in\N^N} \prod_{i=1}^N(k_i-1)^{-(m_i+n_i)/2}
\nonumber\\
&\leq \prod_{i=1}^N \frac{k_i-1}{(\sqrt{k_i-1}-1)^2}
<\infty 
\label{eq:chipc}
\end{align}
as claimed.
\end{proof}

% \[\diamondsuit^{i+1} \leq C \eps \diamondsuit_p \diamondsuit^i, \] 
% from which it will follow that $\diamondsuit_{p+\eps} < \infty$ if $\eps$ is sufficiently small. 

% \begin{prop}
% If $\beta_p>1/2$ then $\nabla_p<\infty$.
% \end{prop}

% % Let $D$ be the degree of $G$. 
% % The BK inequality easily implies that the function
% \begin{proof}
% Let 
% \[K_p(m) = \chi_p(m) - \chi_p(m-1) = \sum_x \tau_p(0,x) \mathbbm{1} ( e^{m-1} \leq \Delta(0,x) \leq e^m). \]
% Then we have that \[K_p(m) \leq e^{-\beta_p m +o(m)}\] for every $m \geq 0$, and the twisted MTP implies that \[K_p(m) \leq e^{-\beta_p m +o(m)} e^{m}\] for $m \leq 0$. 
% Thus, we conclude by observing that
% \begin{align*} \nabla_p &\leq \sum_{x,y,z} \tau_p(0,x)\tau_p(x,y)\tau_p(0,z) \mathbbm{1}(e^{-10} \leq \Delta(0,z) \leq e^{10}) \\&\leq \sum_{m,n} K_p(m) K_p(n) K_p(-m-n),\end{align*}
% which can easily be verified to be finite when $\beta_p>1/2$.
% \end{proof}

\begin{proof}[Proof of \cref{thm:pcpu,thm:main}]
\cref{thm:pcpu} follows immediately from \cref{lem:chipc} and \cref{lem:opencondition}, while \cref{thm:main} follows immediately from \cref{lem:chipc} and \cref{lem:tilttotri}, and in particular the bound \eqref{eq:chipc}.
\end{proof}

\section{A pedestrian proof of Schramm's random walk lemma}
\label{sec:Schramm}

As mentioned in the introduction, Kozma's work \cite{kozma2011percolation} relied upon a lemma of Schramm, which states that if $G=(V,E)$ is a nonamenable transitive unimodular graph and $X$ is the simple random walk on $G$, then 
\begin{equation}
\label{eq:Schramm_original}
\E\left[
\tau_{p_c}(X_0,X_n)
\right] \leq \rho^n,
\end{equation}
where $\rho=\lim_{n\to\infty}p_{2n}(0,0)^{1/2n}$ is the spectral radius of $G$. (The proof of this lemma appears in \cite{kozma2011percolation}.) This shows in particular that connection probabilities for critical percolation on nonamenable Cayley graphs are exponentially small in the distance for at least some choices of vertices (this also follows from \cite[Theorem 1.2]{Hutchcroft2016944}). It is conjectured that they are exponentially small in the distance \emph{uniformly} over all pairs of vertices. 

Schramm's proof of \eqref{eq:Schramm_original} relies on an ingenious use of the mass-transport principle; unimodularity of $G$ and reversibility of $X$ are essential to the argument. In this section, we show that the following more general form of Schramm's Lemma, without these restriction. We also prove a quenched version of Schramm's Lemma that also applies to amenable non-Liouville graphs such as the lamplighter group. The proofs of these generalizations are  obtained from very simple supermultiplicativity considerations, following the same strategy as \cite{Hutchcroft2016944} and the proof of \cref{thm:pcestimate}, and do not use the mass-transport principle. 

\begin{thm} 
\label{thm:Schramm}
Let $G$ be a connected, locally finite graph, let $\Gamma \subseteq \Aut(G)$ be transitive, and let $X$ be a Markov process on $G$ whose transition probabilities $P(x,y)$ are invariant under the diagonal action of $\Gamma$ in the sense that $P(x,y) = P(\gamma x, \gamma y)$ for every $\gamma \in \Gamma$. Then 
\begin{equation}
\label{eq:Schramm_annealed}
\E\Bigl[\tau_{p_c}(X_0,X_n)\Bigr]
% =\sup_{n\geq1}\E[\tau_{p_c}(X_0,X_n)]^{1/n} = 
\leq \left(\lim_{m\to\infty} \biggl( \sup_{y\in V} P^m(x,y) \biggr)^{1/m}\right)^n
\end{equation}
for every $n\geq 1$.
If we have furthermore that $\E\left[d(X_0,X_1)\right] < \infty$, then
\begin{equation}
\label{eq:Schramm_quenched1}
\frac{1}{n}\E \Bigl[\log \tau_{p_c}(X_0,X_n)\Bigr]
% =\exp \sup_{n\geq1}\frac{1}{n}\E[\log\tau_{p_c}(X_0,X_n)]
 \leq  \left[\lim_{m\to\infty} \frac{1}{m}\sum_{y\in V} P^m(x,y) \log P^m(x,y)\right]
 \end{equation}
 for every $n\geq 1$, 
 and
 \begin{equation}
\label{eq:Schramm_quenched2}
\lim_{n\to\infty} \frac{1}{n}\log \tau_{p_c}(X_0,X_n)
= \lim_{n\to\infty}\frac{1}{n}\E \Bigl[\log \tau_{p_c}(X_0,X_n)\Bigr]
% =\exp \sup_{n\geq1}\frac{1}{n}\E[\log\tau_{p_c}(X_0,X_n)]
 \leq \lim_{n\to\infty} \frac{1}{n}\sum_{y\in V} P^n(x,y) \log P^n(x,y) 
 \end{equation}
 almost surely.
\end{thm}

\begin{remark}
The limits appearing on the right-hand sides of \eqref{eq:Schramm_annealed}, \eqref{eq:Schramm_quenched1} and \eqref{eq:Schramm_quenched2} exist by Fekete's Lemma. If $P$ is reversible then the right hand side of \eqref{eq:Schramm_annealed} is its spectral radius, so that Schramm's Lemma follows as a special case of \eqref{eq:Schramm_annealed} by taking $X$ to be the simple random walk on $G$. The limit appearing on the right-hand side of \eqref{eq:Schramm_quenched1} and \eqref{eq:Schramm_quenched2} is exactly the negative of the Avez entropy of $X$; if $X$ is reversible and has $\E [d(X_0,X_1)]<\infty$ then it has non-zero Avez entropy if and only if it has positive asymptotic speed \cite[Theorem 14.20]{LP:book}.
\end{remark}

\begin{remark}
The same proof also yields analogous bounds for the critical free-boundary-condition Ising model.
\end{remark}

\begin{remark}
Despite the excitement generated by Schramm's initial discovery of the inequality \eqref{eq:Schramm_original}, as of yet it has failed to lead to much further progress besides the aforementioned work of Kozma. In fact, we suspect that the bound \eqref{eq:Schramm_original} continues to hold \emph{at $p_u$} for certain transitive nonamenable graphs, which would partly explain this inefficacy. We present our generalization, together with our new proof, primarily as a matter of historical interest.
\end{remark}

\begin{proof}
By the Harris-FKG inequality, we have that
\begin{equation}
\label{eq:Xsubadd}
\tau_p(X_0,X_{n+m})\geq \tau_p(X_0,X_n)\tau_p(X_n,X_{n+m})
\end{equation}
for every $n,m\geq0$. Since $\Gamma$ is transitive and the transition probabilites $P$ are $\Gamma$-invariant, we have that
\[
\E\left[\tau_p(X_0,X_n)\tau_p(X_n,X_{n+m})\right]= \E\left[\tau_p(X_0,X_n) \right] \E\left[\tau_p(X_0,X_m)\right].
\]
We deduce that the sequences 
$\log \E\left[\tau_p(X_0,X_n)\right]$  and $\E\left[\log \tau_p(X_0,X_n)\right]$ 
are both superadditive. Furthermore, we have as before that $\tau_p(x,y)$ is left-continuous in $p$ for each $x,y\in V$, and it follows by dominated convergence that $\E\left[\tau_p(X_0,X_n)\right]$ is left-continuous in $p$ for each $n\geq 1$. Similarly, since $\tau_p(x,y) \geq p^{d(x,y)}$ for every $x,y\in V$, it follows by dominated convergence that $\E\left[\log \tau_p(X_0,X_n)\right]$ is left-continuous in $p$ for each $n\geq 1$ under the assumption that $\E[d(X_0,X_1)]<\infty$, which implies that $\E[d(X_0,X_n)]<\infty$ for every $n\geq 1$.

% By Fekete's Lemma (for the annealed estimate) and the superadditive ergodic theorem (for the quenched estimate), we have
% \[\lim_{n\to\infty}\E[\tau_{p_c}(X_0,X_n)]^{1/n}
% =\sup_{n\geq1}\E[\tau_{p_c}(X_0,X_n)]^{1/n} \]
% % = 
% % \leq X_0 \]
% and
% \[\lim_{n\to\infty}\tau_{p_c}(X_0,X_n)^{1/n}
% =\exp \sup_{n\geq1}\frac{1}{n}\E[\log\tau_{p_c}(X_0,X_n)] \quad \text{a.s.}\]
%  % \leq  e^{-h} \quad \text{a.s.}\]
% The same left continuity argument as before implies that it sufficies to prove the estimates for all $p<p_c$. Let $\chi(p)$ be the susceptibility. Then
Now, observe that if $p<p_c$ then
\begin{align}
\label{eq:SchrammKeyObservation}
\E\left[\frac{\tau_p(X_0,X_n)}{P^n(X_0,X_n)}\right] =  \sum_{x\in V}\tau_p(0,x) \mathbbm{1}\left[ P^n(0,x) >0 \right] \leq \chi_p.
\end{align}
Thus, by \eqref{eq:SchrammKeyObservation} and Fekete's Lemma, we have that
\begin{align*}
\sup_{n\geq1} \E\left[\tau_p(X_0,X_n)\right]^{1/n} &= \lim_{n\to\infty}\E\left[\tau_p(X_0,X_n)\right]^{1/n} \\&\leq \limsup_{n\to\infty} \left[\chi_p \sup_{y\in V} P^n(x,y) \right]^{1/n}=\lim_{n\to\infty} \left[ \sup_{y\in V} P^n(x,y) \right]^{1/n}
\end{align*}
for every $p<p_c$. The claimed inequality \cref{eq:Schramm_annealed} follows by left-continuity of $\E[\tau_p(X_0,X_n)]$. 

Now suppose that $\E[d(X_0,X_1)]<\infty$. Jensen's inequality implies that
\[
\E\left[\log \tau_p(X_0,X_n)\right]\leq \log \E\left[\frac{\tau_p(X_0,X_n)}{P^n(X_0,X_n)}\right] +  \E\left[ \log P^n(X_0,X_n)\right] \leq \log \chi_p +  \E\left[ \log P^n(X_0,X_n)\right],
\]
and it follows by Fekete's Lemma that
\begin{align*}
 \sup_{n\geq1}\frac{1}{n}\E\left[\log \tau_p(X_0,X_n)\right] = \lim_{n\to \infty}\frac{1}{n} \E\left[\log \tau_p(X_0,X_n)\right] \leq \lim_{n\to\infty} \frac{1}{n}\E\left[ \log P^n(X_0,X_n)\right]
\end{align*}
for every $p<p_c$. The inequality \cref{eq:Schramm_quenched1} follows from this together with the left-continuity of $\E\left[\log \tau_p(X_0,X_n)\right]$. The almost sure equality \eqref{eq:Schramm_quenched2} follows from \eqref{eq:Schramm_quenched1}, \eqref{eq:Xsubadd}, ergodicity of random walk in i.i.d.\ random scenery (see e.g.\ \cite[Theorem 4.6]{AL07}), and Kingman's subadditive ergodic theorem.
\end{proof}

\subsection*{Acknowledgements} We thank Aran Raoufi for help with some references and for spotting a typo, and thank the anonymous referee for suggesting various minor improvements.

\footnotesize{
  \bibliographystyle{abbrv}
  \bibliography{unimodularthesis}
  }

\end{document}